\documentclass{amsart}

\title[Frucht's theorem and other set-theoretic principles below $\AC$ and $\AF$]%
      {Frucht's theorem and other set-theoretic principles below the axiom of choice and the axiom of foundation}

\author[Chen]{Junhong Chen}
 \address[Junhong Chen]
         {School of Mathematical Science, Fudan University, 220 Handan Road, Shanghai, 200433 China}
\email{21300180086@m.fudan.edu.cn}

\author[Ju]{Daheng Ju}
 \address[Daheng Ju]
         {Department of Philosophy and Religious Studies, Peking University, 5 Yiheyuan Road, Beijing, 100871 China}
 \email{judaheng@pku.edu.cn}

\usepackage{latexsym,amsfonts,amsmath,amssymb,mathrsfs,tikz-cd,adjustbox}
\usetikzlibrary{fit,decorations.pathmorphing,backgrounds}
\usepackage{changepage,calc}
\usepackage{csquotes}
\usepackage{braket}
\usepackage{bbm}
\usepackage[hidelinks]{hyperref}
\usepackage{relsize}
\usepackage[cmyk,svgnames,dvipsnames]{xcolor}
\usepackage{tikz}
\usetikzlibrary {3d}
\usepackage{comment}
\usetikzlibrary{arrows,arrows.meta,petri,topaths,positioning,shapes,shapes.misc,patterns,calc,decorations.pathreplacing,hobby}
\pgfdeclarelayer{foreground}
\pgfdeclarelayer{background}
\pgfdeclarelayer{boardarrows}
\pgfdeclarelayer{boardgrid}
\pgfdeclarelayer{boardshades}
\pgfsetlayers{boardshades,boardgrid,boardarrows,background,main,foreground}
\usepackage{wrapfig} 
\usepackage{float}
\usepackage[utf8]{inputenc}
\RequirePackage{doi}

\usepackage{enumitem}

\usepackage[backend=bibtex,style=alphabetic,maxbibnames=15,maxcitenames=6,dateabbrev=false]{biblatex}
\renewcommand{\UrlFont}{} 
\renewbibmacro{in:}{\ifentrytype{article}{}{\printtext{\bibstring{in}\intitlepunct}}} 
\DeclareFieldFormat{url}{\UrlFont\url{#1}} 
\DeclareFieldFormat{urldate}{
  (version \thefield{urlday}\addspace%
  \mkbibmonth{\thefield{urlmonth}}\addspace%
  \thefield{urlyear}\isdot)}
\DeclareFieldFormat{eprint:arxiv}{
  \ifhyperref
    {\href{http://arxiv.org/abs/#1}{%
        arXiv\addcolon\nolinkurl{#1}}\iffieldundef{eprintclass}{}{\UrlFont{\mkbibbrackets{\thefield{eprintclass}}}}}
    {arXiv\addcolon\nolinkurl{#1}\iffieldundef{eprintclass}{}{\UrlFont{\mkbibbrackets{\thefield{eprintclass}}}}}}

\renewcommand\emptyset{\varnothing}
\newcommand{\FT}{\mathsf{FT}}
\newcommand{\ZF}{\mathsf{ZF}}

\newcommand{\ZFF}{\mathsf{ZF-AF}}
\newcommand{\ZFFC}{\mathsf{ZF}_{Coll}\mathsf{-AF}}
\newcommand{\ZFCF}{\mathsf{ZFC-AF}}
\newcommand{\AWF}{\mathsf{AWF}}
\newcommand{\AC}{\mathsf{AC}}
\newcommand{\AF}{\mathsf{AF}}
\newcommand{\ARs}{\mathsf{ARs}}
\newcommand{\BAFA}{\mathsf{BAFA}}
\newcommand{\FAFA}{\mathsf{FAFA}}
\newcommand{\SAFA}{\mathsf{SAFA}}
\newcommand{\AFA}{\mathsf{AFA}}
\newcommand{\AIE}{\mathsf{AIE}}
\newcommand{\AR}{\mathsf{AR}}
\newcommand{\ZFA}{\mathsf{ZFA}}
\newcommand{\ZFAC}{\mathsf{ZFA}_{Coll}}
\newcommand{\ZFCA}{\mathsf{ZFCA}}

\newcommand{\RR}{\mathsf{RR}}
\newcommand{\WRR}{\mathsf{WRR}}
\newcommand{\WFT}{\mathsf{WFT}}
\newcommand{\GBF}{\mathsf{GB-AF}}
\newcommand{\GFT}{\mathsf{GFT}}
\newcommand{\GWFT}{\mathsf{GWFT}}

\newcommand{\GC}{\mathsf{GC}}
\newcommand{\DC}{\mathsf{DC}}
\newcommand{\GWO}{\mathsf{GWO}}
\newcommand{\CWFAFA}{\mathsf{CWFAFA}}
\newcommand{\WFAFA}{\mathsf{WFAFA}}
\newcommand{\WF}{\mathord{{\rm WF}}}
\newcommand{\WFP}{\operatorname{wf}}
\newcommand{\IFP}{\operatorname{if}}

\newtheorem{theorem}{Theorem}

\begin{document}

\begin{abstract}
  We take the first step toward the study of set-theoretic principles below the axiom of choice $\AC$ and the axiom of foundation $\AF$ by studying Frucht's theorem, an ordinary mathematical theorem which is provable with either $\AC$ or $\AF$ but not provable without both, and its variants.
  Specifically, we propose a number of such principles, study the relations between these principles and the standard axioms, and prove provability and unprovability results using (infinite) graph-theoretic constructions and permutation models, which draw a preliminary map of this new area of set theory.
\end{abstract}

\maketitle

\tableofcontents

\newcommand\JD[1]{{\color{blue}(JD: #1)}}
\newcommand\CJ[1]{{\color{red}(CJ: #1)}}

\section{Introduction}

Frucht’s theorem, which will be denoted by $\FT$ throughout this paper, says that every group is isomorphic to the automorphism group of
a simple graph. This was proved by Frucht \cite{Frucht1939} in 1939 for finite groups, and generalized independently by deGroot \cite{De_Groot1959-dv} and Sabidussi \cite{Sabidussi1960-dq} to infinite groups in 1959-1960. It is provable in either $\ZF$ or $\ZFCF$(the axiom of foundation). Surprisingly, Pinsky proves in \cite{Pinsky2023-rx} that $\FT$ can't be proved in $\ZFF$ alone(assuming it is consistent), and this is the first time an ordinary mathematical theorem that has such a property is discovered.

There are many theorems of the same style as this original Frucht's theorem, all of which state that every group is isomorphic to the automorphism group of some specified mathematical structure. For example, we may claim that every group is isomorphic to the automorphism group of a poset, which was first stated by Birkhoff in \cite{Birkhoff1946}; another claim, that every group is isomorphic to the autohomeomorphism group of a topological space, was first stated by deGroot in \cite{De_Groot1959-dv}; we can even claim that every group is isomorphic to the automorphism group of a field, and this assertion can be found in \cite{FriedKollar1981}. All of these variants of Frucht's theorem are equivalent with the original Frucht's theorem under $\ZFF$\cite{Kaplan2012-zh}, so we won't spend much time on them.

Throughout this paper, we work in $\ZFF$. Since Frucht's theorem is a graph-theoretic theorem, we shall fix some graph-theoretic conventions first. A graph is a structure $(V,E)$ with $E\subseteq V^2$, and a simple graph is a structure $(V,E)$ with $E\subseteq [V]^2$. A pointed (simple) graph is a (simple) graph with a specified vertex called the point, i.e., a structure $(V,E,p)$. A pointed (simple) graph is accessible if for every vertex there exists a path from the point to it. A structure is rigid if it doesn't have non-trivial automorphisms. We will abbreviate “accessible pointed graph” as APG, and abbreviate ``accessible pointed simple graph'' as APSG. We will abbreviate “rigid APG” as RAPG, and abbreviate ``rigid APSG'' as RAPSG(\cite{Pinsky2023-rx} calls it RCPG).

The proof that $\FT$ is provable from either $\AC$ or $\AF$ is as follows: first, both $\AC$ and $\AF$ imply that every set is in bijection with a set in $\WF$, the well-founded universe (this statement will be denoted by $\AWF$ in this paper); second, $\AWF$ implies that every set is in bijection with a collection of mutually non-isomorphic RAPSGs (this statement will be denoted by $\ARs$ in this paper); finally, $\ARs$ implies $\FT$ according to a classic graph-theoretic construction starting from the Cayley graph of a group.\cite{Pinsky2023-rx}


In this paper, we take the first step toward the study of $\FT$, $\ARs$, $\AWF$ and other set-theoretic principles below $\AC$ and $\AF$.
The structure of this paper is as follows.
In Section~2, we prove that $\FT$ is actually equivalent to $\ARs$, so $\FT$ itself can be viewed as a set-theoretic principle.
Section~3 is devoted to anti-foundation axioms: we show that $\FAFA_2$ implies $\FT$, separate $\FAFA_2$ from both $\AWF$ and $\AC$, and then exhibit that the implication from $\AWF$ and $\FAFA_2$ to $\FT$ is strict.
In Section~4, we collect equivalent forms of $\FT$, including $\AIE$, $\AIE'$, $\AR$, $\ARs$ and $\WRR$.
Section~5 introduces the weak Frucht theorem $\WFT$, where the requirement of rigidity is dropped from $\ARs$, and demonstrates by a permutation model argument that $\WFT$ is strictly weaker than $\FT$.
In Section~6, we move to the second-order set theory $\GBF$ and study globalizations of $\FT$; we show that the global Frucht theorem $\GFT$ coincides with its weak version $\GWFT$, and that a global well-ordering exists exactly when both the axiom of global choice and $\GFT$ hold.
Section~7 is a preliminary study of further set-theoretic principles $\CWFAFA_2$, $\WFAFA_2$ and $\WFAFA_2'$ that lie below $\AWF$ and $\FAFA_2$. The main results of this paper are demonstrated in the following figure.

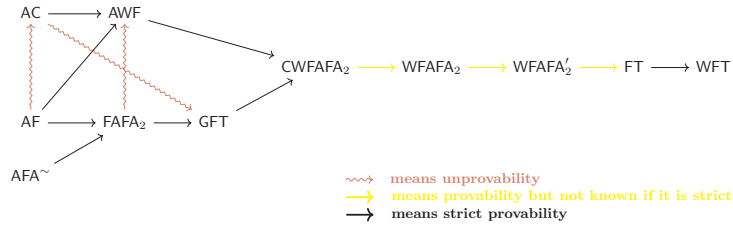
\begin{figure}[H]
\[
\begin{adjustbox}{scale=0.6}
\begin{tikzcd}[
    execute at end picture={
        \definecolor{arrowred}{RGB}{214,92,92}
        \node[
            anchor=south east,
            align=left,
            font=\small\bfseries
        ] at (current bounding box.south east) {
            \textcolor{arrowred}{\tikz[baseline=-0.5ex]{\draw[->, arrowred, decorate, decoration={snake, amplitude=0.3mm, segment length=1.5mm}] (0,0) -- (0.6,0);} \ \ means unprovability} \\
            \textcolor{yellow}{\tikz[baseline=-0.5ex]{\draw[->, yellow, line width=0.4mm] (0,0) -- (0.6,0);} \ \ means provability but not known if it is strict} \\
            \textcolor{black}{\tikz[baseline=-0.5ex]{\draw[->, black, line width=0.4mm] (0,0) -- (0.6,0);} \ \ means strict provability}
        };
    }
]
    \AC & \AWF &&&&&& \\
    &&& \CWFAFA_2 & \WFAFA_2 & \WFAFA_2' & \FT & \WFT \\
    \AF & \FAFA_2 & \GFT &&&&& \\
    \AFA^\sim &&&&&&& \\
    \arrow[from=1-1, to=1-2]
    \arrow[color={rgb,255:red,214;green,92;blue,92}, squiggly, from=1-1, to=3-3]
    \arrow[from=1-2, to=2-4]
    \arrow[color=yellow, from=2-4, to=2-5]
    \arrow[color=yellow, from=2-5, to=2-6]
    \arrow[color=yellow, from=2-6, to=2-7]
    \arrow[from=2-7, to=2-8]
    \arrow[color={rgb,255:red,214;green,92;blue,92}, squiggly, from=3-1, to=1-1]
    \arrow[from=3-1, to=1-2]
    \arrow[from=3-1, to=3-2]
    \arrow[color={rgb,255:red,214;green,92;blue,92}, squiggly, from=3-2, to=1-2]
    \arrow[from=3-2, to=3-3]
    \arrow[from=3-3, to=2-4]
    \arrow[from=4-1, to=3-2]
\end{tikzcd}
\end{adjustbox}
\]
\caption{Main results of this paper}
\end{figure}

\section{Frucht's theorem as a set-theoretic principle}

It's hard to say that $\FT$ is a set-theoretic principle, but it's natural to say that $\ARs$ is a set-theoretic principle, at least it is claiming that every set has some special property.
In this section, we will show that $\FT$ is, in fact, equivalent to $\ARs$.

\begin{theorem}
    $\ZFF\vdash\FT\to\ARs$.
\end{theorem}
\begin{proof}

    For any set $X$, consider the free group $G$ generated by $X\cup\{z\}$, where $z$ is an element not in $X$. Frucht's theorem gives us a simple graph $C=(V,E)$ and an isomorphism $f:G\to\operatorname{Aut}(C)$. Every $x\in X$ can be viewed as an automorphism $f(x):C\to C$, and we only need to encode this automorphism as an RAPSG.

    The construction is as follows. We add to $C$ $5 \times C$ many vertices together with one more vertex $p_{C_x}$ as the point, and add some new edges between these vertices as shown in the following figure (where $a,b$ are vertices of $C$).

    Let us prove that $C_x$ is rigid. Suppose $g$ is an automorphism of $C_x$. First, $g$ must fix $p_{C_x}$. Every vertex $p^1_a$ is adjacent to $p_{C_x}$, so $g$ induces a permutation of the set $\{p^1_a \mid a\in C\}$. If $p^1_a$ is mapped to $p^1_b$, then $p^2_a$, $p^4_a$ and $p^5_a$ are forced to be mapped to $p^2_b$, $p^4_b$ and $p^5_b$, since they have distinct degrees ($2$ for $p^2$, $4$ for $p^4$ and $3$ for $p^5$). Consequently, $g$ must send every $p^i_a$ to $p^i_b$ and send $a$ to $b$. Thus, the restriction of $g$ to $C$ is an automorphism of $C$, and clearly $g$ is completely determined by this restriction.

    Let $f(s)\in\operatorname{Aut}(C)$ be the automorphism of $C$ induced by $g$. Observe that the edge between $p^5_a$ and $f(x)(a)$ and the edge between $p^3_a$ and $f(z)(a)$ must be mapped by $g$ to the edge between $p^5_{f(s)(a)}$ and $f(x)(f(s)(a))$ and the edge between $p^3_{f(s)(a)}$ and $f(z)(f(s)(a))$, respectively. This forces $f(x)(f(s)(a)) = f(s)(f(x)(a))$ and $f(z)(f(s)(a)) = f(s)(f(z)(a))$ to hold for every $a\in C$. Therefore, $f(x)f(s)=f(s)f(x)$ and $f(z)f(s)=f(s)f(z)$, and since $f$ is an isomorphism, $xs=sx$ and $zs=sz$ in the free group $G$. In a free group, $xs=sx$ implies that $s$ is a power of $x$, while $zs=sz$ implies that $s$ is a power of $z$. Because $x$ and $z$ are distinct free generators, the only common power is the identity; thus $s$ is the empty word. Consequently, $f(s)$ is the identity map on $C$, and $g$ is the identity map on $C_x$.

    Now the proof for $C_x$ being non-isomorphic to $C_y$ when $x\neq y$ is of the same manner as above: any such isomorphism must be induced by some $f(s)$ on $C$, and the sequence $s$ must satisfy the equation $sy=xs$, which is impossible.

\end{proof}

\begin{figure}[htbp]
\centering
\adjustbox{scale=0.6,center}{\begin{tikzcd}
	&&&&& {p_{C_x}} &&&&& \\
	&& {p_a^1} &&& \dots &&& {p_b^1} \\
	& {p_a^2} &  &&&&&&  & {p_b^2} \\
	& {p_a^3} & {p_a^4} & {p_a^5} &&&& {p_b^5} & {p_b^4} & {p_b^3} \\
	{f(z)(a)} && a && {f(x)(a)} && {f(x)(b)} && b && {f(z)(b)}
	\arrow[no head, from=1-6, to=2-3]
	\arrow[dotted, no head, from=1-6, to=2-6]
	\arrow[no head, from=1-6, to=2-9]
	\arrow[no head, from=2-3, to=3-2]
	\arrow[no head, from=2-3, to=4-3]
	\arrow[no head, from=2-3, to=4-4]
	\arrow[no head, from=2-9, to=4-9]
	\arrow[no head, from=2-9, to=3-10]
	\arrow[no head, from=2-9, to=4-8]
	\arrow[no head, from=3-2, to=4-2]
	\arrow[no head, from=3-10, to=4-10]
	\arrow[no head, from=4-2, to=5-1]
	\arrow[no head, from=4-3, to=5-3]
    \arrow[no head, from=4-3, to=4-2]
	\arrow[no head, from=4-3, to=4-4]
	\arrow[no head, from=4-4, to=5-5]
	\arrow[no head, from=4-8, to=5-7]
	\arrow[no head, from=4-9, to=5-9]
    \arrow[no head, from=4-9, to=4-8]
	\arrow[no head, from=4-9, to=4-10]
	\arrow[no head, from=4-10, to=5-11]
\end{tikzcd}}
\caption{Construction for $\FT\to\ARs$}
\end{figure}
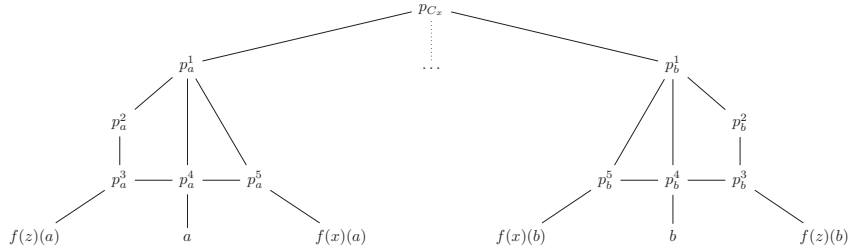

\section{Anti-foundation axioms and Frucht's theorem}

It is known that $\AF$ can prove $\FT$. In non-well-founded set theory, people replace $\AF$ with one of the anti-foundation axioms, which implies the existence of non-well-founded sets. Can they prove $\FT$ too? In this section, we will answer this question. Specifically, we will show that a weakening of $\AF$ called $\FAFA_2$ implies $\FT$, while among the four well-known anti-foundation axioms, $\BAFA$, $\FAFA$, $\SAFA$, and $\AFA$, all the latter three are of the type $\mathsf{AFA}^\sim$ (where $\sim$ is one of the regular
bisimulations), and then imply $\FAFA_2$ just as $\AF$ does\cite{Aczel1988-vp}.

$\FAFA_2$ is the following assertion: the canonical picture of a set must be extensional and isomorphism-extensional. Here, the canonical picture of a set $x$ is the APG $(\operatorname{trcl}(\{x\}),\ni,x)$. Given an APG $G=(V_{G},E_{G},p_{G})$ and a vertice $a\in V$, the set of children of $a$ is $C_G(a)=\{b\in V_{G}\mid a\to b\}$ and the descendant subgraph of $a$ is $G[a]=(V_{G[a]},E_{G}\cap V_{G[a]}^2,a)$, where $V_{G[a]}=\{b\in V_{G}\mid $there is a path from $a$ to $b\}$. An APG is extensional if $C_G(a)=C_G(b)\to a=b$, and isomorphic-extensional if $G[a]\cong G[b]\to a=b$. Clearly, $\AF\to\FAFA_2$ follows from Mostowski's collapsing lemma.

$\FAFA_2\to\FT$ is shown by the following weakening chain from $\FAFA_2$ to $\ARs$.
\begin{enumerate}
    \item $\AIE$, the assertion that every set is in bijection with a collection of mutually non-isomorphic extensional, isomorphism-extensional APGs.

    It is weaker than $\FAFA_2$: just replace every element of the set with its canonical picture, and the fact that they are mutually non-isomorphic followed from an equivalent form of $\FAFA_2$ which asserts that two sets having isomorphic canonical pictures are identical.
    \item $\AIE'$, the assertion that every set is in bijection with a collection of mutually non-isomorphic isomorphism-extensional APGs.
    
    It is obviously weaker than $\AIE$.
    \item $\AR$, the assertion that every set is in bijection with a collection of mutually non-isomorphic RAPGs.

    It is weaker than $\AIE'$ because isomorphism-extensional APGs are rigid due to a standard argument considering the descendant subgraphs of the exchanged vertices.
\end{enumerate}, 

Finally, $\ARs$ is weaker than $\AR$: we can convert an RAPG to an RAPSG using the well-known construction below, which also occurs in \cite{Pinsky2023-rx}.

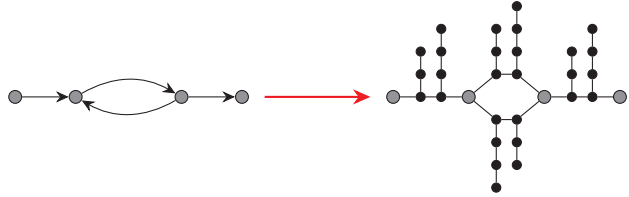
\begin{figure}[htbp]
\centering
\begin{tikzpicture}[
    >={Stealth[width=1.5mm]},
    old/.style ={circle, draw, fill=gray, inner sep=0pt, minimum size=1.7mm},
    new/.style ={circle, draw, fill, inner sep=0pt, minimum size=1.2mm},
]
\begin{scope}[xshift=-3cm]
    \node[old] (a) at (0,0) {};
    \node[old] (b) at (.8,0) {}
        edge[<-] (a);
    \node[old] (c) at (2.2,0) {}
        edge[->, bend left = 30] (b)
        edge[<-, bend right = 30] (b);
    \node[old] (d) at (3,0) {}
        edge[<-] (c);
\end{scope}
\draw[-{Stealth[width=2mm]}, red, thick] (.3, 0) -- (1.7, 0);
\newcommand{\drawstick}[5]{
    \foreach \x [remember=\x as \prev (initially 1)] in {1,...,#2}{
        \node[new] (c#1_\x) at ($(#1)+\x*(#3:#4)+(#3:#5)$) {}
            \ifnum \x > 1
                edge (c#1_\prev)
            \fi;
    }
    \draw (c#1_1) to (#1);
}
\newcommand{\drawDirEdge}[4]{
    \path (#1)--(#2)
        node[new, pos=1/3, yshift = #3] (int1) {}
        node[new, pos=2/3, yshift = #3] (int2) {};
    \drawstick{int1}{2}{#4}{.3}{0};
    \drawstick{int2}{3}{#4}{.3}{0};
    \draw (#1) -- (int1) -- (int2) -- (#2);
}
\begin{scope}[xshift = 2cm]
    \node[old] (a) at (0,0) {};
    \node[old] (b) at (1,0) {};
    \node[old] (c) at (2,0) {};
    \node[old] (d) at (3,0) {};
    \drawDirEdge{a}{b}{0}{90}
    \drawDirEdge{b}{c}{.3cm}{90}
    \drawDirEdge{c}{b}{-.3cm}{270}
    \drawDirEdge{c}{d}{0}{90}
\end{scope}
\end{tikzpicture}
\caption{Convert RAPG to RAPSG\cite{Pinsky2023-rx}}
\end{figure}

Obviously, $\AWF$ implies all these weakenings, so we give a new proof of $\AWF\to\FT$ as well. 

\begin{theorem}
    $\ZFF\vdash\AWF\to\AIE$.
\end{theorem}
\begin{proof}
    $\AWF$ bijectively maps every set to a set in $\WF$. Since $\WF\vDash\AIE$ (followed from $\WF\vDash\FAFA_2$), $\AIE$ holds in the whole universe.
\end{proof}

Naturally, one may ask if $\AIE$ can prove $\AWF$ or $\FAFA_2$, and the answer is negative. In fact, we can prove that $\AWF$ and $\FAFA_2$ are independent over $\ZFF$, and $\AIE$ is weaker than both of them becomes an obvious corollary.

\begin{theorem}
    If $\ZFF$ is consistent, it can't prove $\AC\to\FAFA_2$.
\end{theorem}
\begin{proof}
    Consider the canonical model of $\ZFCF+\BAFA$\cite{Aczel1988-vp}. It's a model of $\neg\FAFA_2$, since $\FAFA_2$ contradicts $\BAFA$.
\end{proof}

However, for readers who might not be so familiar with non-well-founded set theory, we provide another proof using techniques from set theory with atoms. The following transfer theorem is useful because we can use permutation models under $\ZFA$, and we need that to construct counterexample models and prove independence results.

\begin{theorem}\label{thm:transfer}
    Given a model $(V,\in,A)$ of $\ZFA$ and $R$, a definable set-like relation on $A$ inside the model. Suppose $R$ is extensional and ill-founded, which means $\forall a,b(\forall c(cRa\leftrightarrow cRb)\to a=b)$ and $\forall a\exists b(bRa)$. We can construct a model $(W,\bar{\in})$ of $\ZFF$ and a surjection $\pi:V\to W$ respecting the $\in$ relation, such that $\pi$ is an isomorphism between the pure-part of $V$ and the well-founded part of $W$, an injection on $A$, and for every atom $a\in A$ we have $\pi(a)=\{\pi(b)\mid bRa\}^W$. If $V$ is a model of $\ZFAC$ then we get a model of $\ZFFC$, if $V$ satisfies $\AC$ then we get a model of $\AC$(the assertion that every set is in bijection with an ordinal).
\end{theorem}

 Notice that, as usual, we assume that $\ZFA$ and $\ZFF$ are both formed with the replacement schema throughout this paper, while it's widely known that in these systems replacement can't prove the collection principle\cite{Yao2026-rc}. So we denote the systems that use the stronger collection instead of replacement by $\ZFAC$ and $\ZFFC$ in our statement of this transfer theorem.

\begin{proof}
    Modify the proof in \cite{Brian2013-jz}.
\end{proof}

Now we can present another proof of the claim proved above.
\begin{proof}
    Consider the $\ZFCA$ model with exactly two atoms $a,b$ and $R=\{(a,a),(b,b)\}$ and use the transfer theorem. Now, the converted model has two different Quine atoms $\pi(a)$ and $\pi(b)$, contradicting $\FAFA_2$. 
\end{proof}

The next proof also makes use of the transfer theorem.

\begin{theorem}
    If $\ZFF$ is consistent, it can't prove $\FAFA_2\to\AWF$.
\end{theorem}
\begin{proof}
    Consider the Ordered Mostowski Model in the context of permutation models. It could be transferred to a model of $\ZFF$ by directly saying that atoms are sets and the less than relation over atoms is the $\in$ relation(and that's why we need the ordering over atoms).

    First, $\AWF$ fails in our converted model. Notice that all the previous atoms in the converted model are non-well-founded, so a set is in the well-founded universe of our converted model if and only if it is a pure set in Ordered Mostowski Model. However, we know that the universe of pure sets in any permutation model must satisfy the axiom of choice, and thus is in bijection with an ordinal. If $\AWF$ holds in our converted model, then every set in the Ordered Mostowski Model will be in bijection with a pure set, thus in bijection with an ordinal, and this is actually equivalent to claiming that the Ordered Mostowski Model satisfies the axiom of choice, an obvious contradiction.

    Standard results on the Ordered Mostowski Model say that for every two atoms $a\neq b$, the ordered structure $(a,\leq)$ and $(b,\leq)$ are never isomorphic; a detailed proof of it could be found in \cite{Pinsky2023-rx}. Now we know that the canonical pictures of all the previous atoms in our converted model are all extensional and isomorphism-extensional: these previous atoms are transitive, so the canonical picture of an atom only adds itself to the initial segment of it; extensionality comes from the fact that the linear order is still dense, and isomorphism-extensionality is just because there is no isomorphism between two different initial segments.

    Now we claim that $\FAFA_2$ holds in our converted model. The property of being converted from an atom is definable in our model: a set has this property if and only if it is transitive and the $\in$ relation on it is dense. So we can define the pseudo-rank function in our model, which is converted from the rank function in the Ordered Mostowski Model and only maps the empty set and all the atoms to $0$. Then an easy induction shows that two canonical pictures are isomorphic if and only if the ``non-well-founded'' parts and the ``well-founded'' parts are respectively isomorphic, but the partial isomorphism over the non-well-founded parts is impossible unless these parts actually coincide, and then the partial isomorphism over the well-founded parts will actually be the identity map. So, two canonical pictures are isomorphic if and only if they are identical, which is an equivalent form of $\FAFA_2$.
\end{proof}

\section{Equivalent forms of Frucht's theorem}

The section above shows that the two weakening chains from $\AWF$ and $\FAFA_2$ to $\ARs$, or equivalently to $\FT$, coincide after the first step $\AIE$, which is strictly weaker than both starting principles.
In this section, we will prove that the rest of this long weakening chain is actually a series of equivalent forms of $\FT$.
The strategy is to make several constructions and reverse the directions of the implications one by one.

\begin{theorem}
    $\ZFF\vdash\ARs\to\AR$.
\end{theorem}
\begin{proof}
    Replace each undirected edge $\{a,b\}$ with two directed edges $(a,b)$ and $(b,a)$.
\end{proof}

\begin{theorem}
    $\ZFF\vdash\AR\to\AIE'$.
\end{theorem}
\begin{proof}
    We describe the operation presented below that converts an RAPG into an isomorphism-extensional APG. Suppose that the original RAPG has the point $p$ and all other vertices $a,b,\dots$. We will add two new vertices, $P$ and $q$, to form the set of vertices in the new APG, and let $P$ be the point.

    In the new APG, $q$ is the only vertex that has no child, and $P$ is the only vertex that points to $q$, so the descendant subgraphs of them are obviously not isomorphic to any other descendant subgraph.
    The descendant subgraph of every vertex except $q$ is actually the whole APG, because every vertex except $q$ and $P$ points to $P$, and then $p,q$ are both accessible and then everything follows.
    Now suppose to the contrary that the descendant subgraphs of $a$ and $b$ are isomorphic. Such an isomorphism must fix the vertices $q$ and $P$, so it must fix the vertex $p$ since $p$ is the only vertex pointed by $P$ other than $q$. Thus, it induces an automorphism of the original RAPG that commutes $a$ and $b$, a contradiction.

    The same reason explains why two non-isomorphic RAPGs are still non-isomorphic after the operation.
\end{proof}

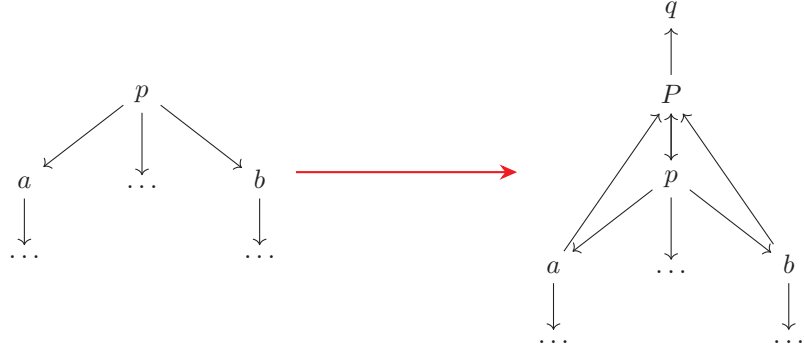
\begin{figure}[htbp]
\centering
\begin{tikzpicture}[>={Stealth[width=1.5mm]}]
  \node (left) {
    \begin{tikzcd}[ampersand replacement=\&]
      \& p \& \\
      a \& \dots \& b \\
      \dots \&\& \dots
      \arrow[from=1-2, to=2-1]
      \arrow[from=1-2, to=2-2]
      \arrow[from=1-2, to=2-3]
      \arrow[from=2-1, to=3-1]
      \arrow[from=2-3, to=3-3]
    \end{tikzcd}
  };
  \node (right) at (7,0) {
    \begin{tikzcd}[ampersand replacement=\&]
	\& q \& \\
	\& P \\
	\& p \\
	a \& \dots \& b \\
	\dots \&\& \dots
	\arrow[from=2-2, to=1-2]
	\arrow[from=2-2, to=3-2]
    \arrow[from=3-2, to=2-2]
	\arrow[from=3-2, to=4-1]
	\arrow[from=3-2, to=4-2]
	\arrow[from=3-2, to=4-3]
	\arrow[from=4-1, to=2-2]
	\arrow[from=4-1, to=5-1]
	\arrow[from=4-3, to=2-2]
	\arrow[from=4-3, to=5-3]
\end{tikzcd}
  };
  \draw[-{Stealth[width=2mm]}, red, thick] (left.east) -- (right.west);
\end{tikzpicture}
\caption{Convert RAPG to isomorphism-extensional APG}
\end{figure}

\begin{theorem}
    $\ZFF\vdash\AIE'\to\AIE$
\end{theorem}
\begin{proof}
    If an isomorphism-extensional APG is not extensional, consider the following two kinds of vertices in it: a vertex of the first kind has only one child, which is itself, and the others form the second kind. Obviously, vertices of the first kind won't have the same set of children as any other vertices, so we just need to deal with vertices of the second kind. For every vertex $a$ of the second kind, we add a new vertex $p_a$ and three new edges: $p_a\to p_a,a\to p_a,p_a\to a$. Now, $p_a$ points to $p_a$ and $a$, and $a$ either points to $p_a$ but not to $a$ itself, or points to $p_a$, $a$, and some other vertices. Thus, the new APG is clearly extensional. It's direct to check that this operation doesn't destroy isomorphism-extensionality and mutual non-isomorphism.
\end{proof}

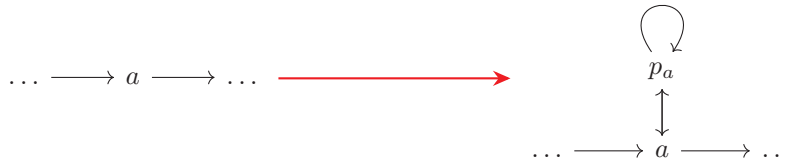
\begin{figure}[htbp]
\centering
\begin{tikzpicture}[>={Stealth[width=1.5mm]}]
  \node (left) {
\begin{tikzcd}
	\dots & a & \dots
	\arrow[from=1-1, to=1-2]
	\arrow[from=1-2, to=1-3]
\end{tikzcd}
  };
  \node (right) at (7,0) {
\begin{tikzcd}
	& {p_a} & \\
	\dots & a & \dots
	\arrow[from=1-2, to=1-2, loop, in=55, out=125, distance=10mm]
	\arrow[from=1-2, to=2-2]
	\arrow[from=2-1, to=2-2]
	\arrow[from=2-2, to=1-2]
	\arrow[from=2-2, to=2-3]
\end{tikzcd}
  };
  \draw[-{Stealth[width=2mm]}, red, thick] (left.east) -- (right.west);
\end{tikzpicture}
\caption{Extensionalize an isomorphism-extensional APG}
\end{figure}

Hamkins and Palumbo consider the weak choice principle $\RR$ in \cite{Hamkins2012-vm}, which states that every set admits a rigid binary relation on it.
If we only need every set to be embedded into the set of vertices of a rigid graph, this principle becomes weaker and is provable in $\ZF$ since $\AF$ ensures that $\operatorname{trcl}(\{x\})$ always carries an RAPG structure.
We will denote this natural weakening by $\WRR$.
In fact, it is also equivalent to $\FT$.

\begin{theorem}
    $\ZFF\vdash\AR\leftrightarrow\WRR$.
\end{theorem}
\begin{proof}
    If $\AR$ holds, any set $X$ is bijectively mapped to a collection $\{R_x\mid x\in X\}$ of mutually non-isomorphic RAPGs. Now add a new vertex that points to the vertices of every $R_x$ and itself, this large graph is obviously rigid: we must fix the new vertex because it's the only vertex such that every vertex is accessible from it; then, we can't exchange the points of two different $R_x$ and $R_y$ because this requires an isomorphism between $R_x$ and $R_y$; finally, we must fix every vertex in every $R_x$ because every $R_x$ is rigid.

    If $\WRR$ holds, fix an injection $f:X\to V$ where $V$ is the set of vertices of a rigid graph $G=(V,E)$. For every $x\in X$, consider the following APG $G_x$: the point is a new vertex $p_x$, the vertices are $\{p_x\}\cup V$, and the edges are the edges of $V$ together with $p_x\to v$ for every $v\in V$ and $f(x)\to p_x$. Now $G_x$ is rigid because if an automorphism fixes $p_x$, then the rest of this automorphism is just an automorphism of $G$, so it must fix all $v\in V$ because $G$ is rigid. Similarly, $G_x$ is not isomorphic to $G_y$ because if such an isomorphism exists, it will induce an automorphism of $G$ that exchanges $f(x)$ and $f(y)$, which is impossible.
\end{proof}

\section{Variants of Frucht's theorem I: weak Frucht's theorem}

Our discussions above show that $\FT$ is equivalent to each of the following assertions: every set is in bijection with a mutually non-isomorphic collection of extensional and isomorphism-extensional APGs, or isomorphism-extensional APGs, or RAPGs, or RAPSGs.
A natural weakening is to assert that every set is in bijection with a mutually non-isomorphic collection of structures satisfying fewer requirements, and this idea actually gives rise to another family of equivalent principles.

\begin{theorem}
    The following four assertions are equivalent under $\ZFF$, and we will name them as the weak Frucht's theorem $\WFT$:
    \begin{enumerate}
        \item Every set is in bijection with a mutually non-isomorphic collection of graphs;
        \item collection of simple graphs;
        \item collection of APGs;
        \item collection of APSGs.
    \end{enumerate}
\end{theorem}
\begin{proof}
    The way we convert RAPGs to RAPSGs can be used to convert mutually non-isomorphic graphs and APGs to mutually non-isomorphic simple graphs and APSGs; and to convert mutually non-isomorphic simple graphs and APSGs to mutually non-isomorphic graphs and APGs, we still just need to replace each undirected edge $\{a,b\}$ with two directed edges $(a,b)$ and $(b,a)$.

    To convert a graph to an APG and preserve the property being mutually non-isomorphic, we just add a new vertex as the point pointing to every vertex in the previous graph. To convert an APG $G=(V,E,p)$ to a graph and preserve the property being mutually non-isomorphic, we add two new vertices $p_1,p_2$ and new edges $p\to p_1,p_1\to p_1,p_1\to p_2,p_1\to v$ for every $v\in V$. Now, $p_1$ is the only vertex such that every vertex is pointed by it, $p$ is the only vertex pointing to $p_1$ that is not $p_1$ itself, $p_2$ is the only vertex such that every path from $p$ to it must go through $p_1$, so every isomorphism between two such graphs must induce an isomorphism between the previous APGs, which doesn't exist.
\end{proof}

Next, we prove that $\WFT$ is strictly weaker than $\FT$. Notice that the proof that $\FT$ can't be proved in $\ZFF$ in \cite{Pinsky2023-rx} also works for $\WFT$, so $\WFT$ is a new set-theoretic principle below $\AC$ and $\AF$.
\begin{theorem}
    If $\ZFF$ is consistent, it can't prove $\WFT\to\FT$.
\end{theorem}
\begin{proof}
    The strategy is to construct a model of $\WFT\wedge\neg\FT$ in the language of set theory with (possibly class many) atoms, then the transfer theorem provides our desired model of $\ZFF+\WFT+\neg\FT$.

    Consider the following model $V$ of $\ZFCA$: the set of atoms $A$ is countable and a bijective enumeration $\{a_s\mid s\in\omega^{<\omega}\}$ is fixed. Let $\operatorname{ht}(a_s)=\operatorname{dom}(s)$, $A_n=\{a_s\mid \operatorname{dom}(s)=n\}$, $A_{\leq n}=\bigcup_{m\leq n}A_m$, $p:A\backslash\{a_{\emptyset}\}\to A$ map $a_{s}$ to $a_{s\upharpoonright(\operatorname{dom}(s)-1)}$, $B(a)=p^{-1}(\{a\})$, $a_s\preceq a_t$ if and only if $s\subseteq t$, and $T_a=\{b\mid a\preceq b\}$.

    Consider $\mathcal{G}=\operatorname{Aut}(A,p)$, the group of all automorphisms of $A$ that respect $p$. Now, the standard definition gives $\operatorname{Fix}(S)=\{\pi\in\mathcal{G}\mid\pi\upharpoonright S=\operatorname{id}_S\}$. We let $\mathcal{F}$ be the normal filter of subgroups generated by all $\operatorname{Fix}(S)$ with $S\in[A]^{<\omega}$. So what we want first is a finite-support permutation model $W=\operatorname{HS}(V,\mathcal{G},\mathcal{F})$, which satisfies $\ZFA$ and contains the set of all atoms $A$ in it.

    Now live in $W$, consider $\mathcal{I}=\{Z\subseteq A\mid \exists n<\omega(Z\subseteq A_{\leq n})\}$ which is an ideal on $A$. Let the small-kernel model $M=\{x\in W\mid \operatorname{ker}(x)\in\mathcal{I}\}$. In \cite{Yao2026-rc}, we have the standard small-kernel theorem claiming that $M$ is a model of $\ZFA$. Notice that all atoms form a proper class in $M$.

    In $W$, we define APG $C_a=(\{a\}\cup B_a,E_a,a)$ for every atom $a$, where $E_a=\{a\to b\mid b\in B_a\}$. We claim that $C_a$ is not isomorphic to any $C_b$ in $W$ if $a\neq b$; in fact, we can prove that there is no injection from $B_a$ to $B_b$. Assume to the contrary that $h:B_a\to B_b$ is an injection, pick a finite support $S$ of $h$; since $T_u(u\in B_a)$ are pairwise disjoint, for every $e\in S\cup\{b\}$ there is at most one $u\in B_a$ such that $e\in T_u$. So we can pick $u,v\in B_a$ such that $(T_u\cup T_v)\cap(S\cup\{b\})=\varnothing$. Now let $\pi$ be the automorphism that only exchanges $T_u$ and $T_v$ but fixes everything else; $S\cup B_b$ will be fixed pointwise by $\pi$, so $\pi(h)=h$ and $h(v)=h(\pi(u))=\pi(h)(\pi(u))=\pi(h(u))=h(u)$, contradicting the fact that $h$ is an injection.

    Now, we have proved that in $W$, every set of atoms in $M$ is in bijection with a mutually non-isomorphic collection of APGs in $M$, and obviously in $M$ this is true as well. By modifying the proof of Lemma 3 in \cite{Pinsky2023-rx}, we can prove that if the set of atoms in $\operatorname{trcl}(\{x\})$ is in bijection with a mutually non-isomorphic collection of APGs, then so is the set $x$. As a consequence, $M\vDash\WFT$.

    Let's work in $V$ to derive more results on $W$. We claim that, given $m<\omega$, $\pi\in\operatorname{Fix}(A_{\leq m})$, and $E\in[A]^{<\omega}$, there is a $\tau\in \operatorname{Fix}(A_{\leq m})$ such that $\tau\upharpoonright E=\pi\upharpoonright E$ and $F_\tau=\{u\in A_{m+1}\mid \tau(u)\neq u\}$ is finite. To prove this claim, notice that $R=\{a_{e\upharpoonright m+1}\mid e\in E, \operatorname{ht}(e)\geq m+1\}$ is a finite subset of $A_{m+1}$. $\pi$ fix $A_{\leq m}$, so it is a permutation on every $B_a\subseteq A_{m+1}$. For every such $B_a$ intersecting $R$, we pick a minimal permutation on $B_a$ generated by $\pi\upharpoonright (B_a\cap R)$. Gluing all these permutations together and extending them naturally gives us the desired $\tau$.

    Now we prove that in $W$, there is no injection $j$ from $A_1$ to a graph $D=(V_D,E_D)$ such that $D$ is rigid and $\operatorname{ker}(D\cup\{j\})\subseteq A_{\leq N}$ for some $N$. Assume to the contrary that these objects exist, let $S$ be a finite support of $D\cup\{j\}$. For $1\leq m\leq N$, consider the assertion $\mathrm{P}_m$ which states that every $\pi\in\operatorname{Fix}(A_{\leq m})\cap\operatorname{Fix}(S)$ fixes $V_D$ pointwise. We will prove that all $\mathrm{P}_m$ holds.
    
    \begin{itemize}
        \item $\mathrm{P}_N$ holds since any such $\pi$ must fix $\operatorname{ker}(D)$ pointwise.
        \item Now we prove $\mathrm{P}_m$ from $\mathrm{P}_{m+1}$. For $\pi\in\operatorname{Fix}(A_{\leq m})\cap\operatorname{Fix}(S)$ and $v\in V_D$, let $T$ be a finite support of $v$ and use the previous observation on $\pi$ and $S\cup T$ to get $\tau\in\operatorname{Fix}(A_{\leq m})$ such that $\tau\upharpoonright(S\cup T)=\pi\upharpoonright(S\cup T)$ and $F=\{u\in A_{m+1}\mid \tau(u)\neq u\}$ is finite. $\pi$ fixes $S$ pointwise, so $\tau$ also does. Thus, $\tau(D)=D$.

        Now consider $\rho_\tau:V_D\to V_D$ such that $\rho_\tau(w)=\tau(w)$ for every $w\in V_D$. We claim that $\rho_\tau\in W$ because $S\cup F$ is one of its finite supports. For any $\sigma\in\operatorname{Fix}(S\cup F)$, let $\delta=\sigma\tau\sigma^{-1}\tau^{-1}$ be the commutator; $\tau$ fixes $A_{\leq m}$ pointwise, so $\delta$ also does; at level $A_{m+1}$, $\tau$ only moves atoms in $F$, which will be fixed pointwise by $\sigma$, so again $\delta$ fixes $A_{m+1}$ pointwise. Finally, $\delta\in\operatorname{Fix}(A_{\leq m+1})\cap\operatorname{Fix}(S)$ because $\sigma$ and $\tau$ both fix $S$ pointwise. The induction hypothesis now gives us that $\delta$ fix $V_D$ pointwise, or $\sigma(\rho_\tau)=\rho_\tau$.

        In $W$, $\rho_\tau$ is an automorphism of $D$, so rigidity of $D$ ensures $\rho_\tau=\operatorname{id}$, and then $\tau(v)=v$; $\pi^{-1}\tau$ fixes $T$ pointwise, so $\pi(v)=v$, and the proof is finished.
    \end{itemize}

    Now $\mathrm{P}_1$ holds, i.e., everything in $\operatorname{Fix}(A_1)\cap\operatorname{Fix}(S)$ fixes $V_D$ pointwise. Let $R=\{a\in A_1\mid S\cap T_a\neq\emptyset\}$ which is a finite set; take $a,b\in A_1\backslash R$ and construct an automorphism $\eta$ that exchanges $T_a,T_b$ and fixes everything else. Again, we prove that the finite set $S\cup\{a,b\}$ supports $\rho_\eta:V_D\to V_D,w\mapsto \eta(w)$, so $\rho_\eta\in W$. Again, let $\sigma\in\operatorname{Fix}(S\cup\{a,b\})$ be arbitrary and $\delta=\sigma\eta\sigma^{-1}\eta^{-1}$ be the commutator. Again, $\delta$ fixes $A_1$ and $S$ pointwise, so $\mathrm{P}_1$ ensures that $\delta$ fixes $V_D$ pointwise. Again, $\rho_\eta\in W$, so by rigidity it is $\operatorname{id}$. Now $S$ supports $j$ but $\eta$ fixes $S$, so $\eta(j)=j$ and $j(b)=\eta(j)(\eta(a))=\eta(j(a))=\rho_\eta(j(a))=j(a)$, contradicting the fact that $j$ is an injection.

    If $M\vDash\FT$, we have $M\vDash\WRR$, which means that there is an injection $j$ from $A_1$ into the set of vertices of a graph $D$ that is rigid in $M$. Obviously $D$ is still rigid in $W$ and $j$ is still an injection in $W$, but $D,j\in M$ makes $\operatorname{ker}(D\cup\{j\})\subseteq A_{\leq N}$ for some $N<\omega$, contradicting the claim proved above.
\end{proof}


\section{Variants of Frucht's theorem II: global Frucht's theorem}

Another way to get variants of Frucht's theorem is to globalize the existence of a bijection from every set to the existence of a bijection from the whole universe, which must be stated in second-order set theory.
So we will work in $\GBF$ in this section, formed by removing the axiom of foundation from G\"odel-Bernays set theory (second-order, without any axiom of choice).

Naturally, we will have a globalized version of Frucht's theorem and a globalized version of the weak Frucht's theorem.
$\GFT$ is the assertion that there is a class function defined on the universe of all sets such that every set is assigned with an RAPG that is unique up to isomorphism.
$\GWFT$ is the assertion that there is a class function defined on the universe of all sets such that every set is assigned with an APG that is unique up to isomorphism.
Notice that all the proofs about different equivalent forms of $\FT$ and $\WFT$ still work in $\GBF$, we can switch the kind of structures assigned to sets as we want.
Surprisingly, unlike that $\WFT$ is strictly weaker than $\FT$, their globalized versions are equivalent.
\begin{theorem}
    $\GBF\vdash\GFT\leftrightarrow\GWFT$.
\end{theorem}
\begin{proof}
    Obviously $\GFT\to\GWFT$, so we will construct a class function $x\mapsto D(x)$ assigning each set with an RAPG that is unique up to isomorphism, provided with $x\mapsto C(x)$ where $C(x)$s are mutually non-isomorphic APGs.
    
    Suppose $C(x)=(V(x),E(x),p(x))$.
    We recursively construct a sequence of APGs. For $C_1(x)$, the set of vertices would be $\{\omega\}\cup\omega\cup\{(1,x,a)\mid a\in V(x)\}$, and the set of edges would be $\{\omega\to n\mid n<\omega\}\cup\{n\to m\mid m<n<\omega\}\cup\{0\to (1,x,p(x))\}\cup\{(1,x,a)\to(1,x,b)\mid(a,b)\in E(x)\}\cup\{1\to (1,x,a)\mid a\in V(x)\}$. $\omega$ will be the point. Now for $C_k(x)(k\geq 2)$, the construction will be that, for every $v\in V(C_{k-1}(x))$ of the form $(k-1,u,a)$, we add $\{(k,v,b)\mid b \in V(v)\}$ as new vertices, and add $\{v\to(k,v,p(v))\}\cup\{(k,v,b)\to(k,v,c)\mid(b,c)\in E(v)\}\cup\{k\to(k,v,b)\mid b\in V(v)\}$ as new edges. Since we are constructing step by step, $C_k(x)$ is naturally a subgraph of $C_{k+1}(x)$, so taking the union after $\omega$ steps we get the desired $D(x)$.

    Let us show that $D(x)$ and $D(y)$ are not isomorphic once $x\neq y$. If there is an isomorphism from $D(x)$ to $D(y)$, $\omega$ will be mapped to $\omega$, then every $n$ will be mapped to $n$ clearly. Now, the children of $1$ besides $0$ in the two APGs form a copy of $C(x)$ and a copy of $C(y)$, and $(1,x,p(x))$ must be mapped to $(1,y,p(y))$ as the only vertex pointed by $0$. Thus, $C(x)$ is isomorphically mapped to $C(y)$, which is impossible.

    Let us show that $D(x)$ is rigid. Similarly, any automorphism must fix $\omega$ and every $n$. If the automorphism is not the identity function, we may pick the least $n$ such that the automorphism does not fix some $(n,u,a)$; it is a child of $n$, so it must be converted to some other $(n,v,b)$. The descendants of $(n,u,a)$ pointed by $n+1$ form a copy of $C((n,u,a))$, and the descendants of $(n,v,b)$ pointed by $n+1$ form a copy of $C((n,v,b))$. Thus, the automorphism induces an isomorphism between $C((n,u,a))$ and $C((n,v,b))$, which is impossible.
\end{proof}

Now we still have that $\GFT$ is strictly weaker than $\FAFA_2$ easily.
\begin{theorem}
    $\GBF\vdash\FAFA_2\to\GFT$.
\end{theorem}
\begin{proof}
    Take the canonical pictures.
\end{proof}

\begin{theorem}
    If $\GBF$ is consistent, it can't prove $\GFT\to\FAFA_2$.
\end{theorem}
\begin{proof}
    Consider the model in the second proof of $\ZFF\not\vdash\AC\to\FAFA_2$.
\end{proof}

For the direction from choice principles, since we're working under second-order set theory, the natural question becomes: can the axiom of global choice $\GC$, which asserts the existence of a global choice function, prove $\GFT$? Here we must distinguish $\GC$ from the axiom of global well-ordering $\GWO$, since the latter is strictly stronger than the former without $\AF$\cite{Howard1978-je}.

\begin{theorem}
    $\GBF\vdash\GWO\to\GFT$, and if $\GBF$ is consistent, it can't prove $\GC\to\GFT$.
\end{theorem}
\begin{proof}
    Actually, $\GBF\vdash\GWO\leftrightarrow(\GC\wedge\GFT)$, and it's sufficient to prove it.
    
    Now, $\GWO\to\GC$ is well-known, and $\GWO\to\GFT$ follows from the fact that ordinals are naturally RAPGs. For $\GC\wedge\GFT\to\GWO$, suppose that $\GFT$ and $\GC$ hold together. A graph in the statement of $\GFT$ can be encoded as a graph on the least possible cardinal(use $\GC$ to pick one; notice that we already have $\AC$ so it is possible), and such a graph can be viewed as $<$ another if it has smaller cardinality or the cardinality is the same but it is earlier on this cardinal (again use $\GC$ to pick a local well-ordering of all structures on each cardinal), which gives us a global well-ordering.
\end{proof}

\section{More principles below $\AC$ and $\AF$}

A natural way to weaken $\FAFA_2$, which has an equivalent form claiming that every set can be injectively represented by an extensional and isomorphism-extensional APG, is to claim that every set can be represented by an extensional and isomorphism-extensional APG. We will denote this assertion by $\WFAFA_2$.
Here, an APG $G$ represents a set $X$ if $G$ represents the canonical picture of $X$, and an APG $G$ represents another APG $G'$ if there exists a decoration $d$ from $G$ to $G'$, where a map $d$ between APGs is a decoration if and only if $d$ maps the point to the point and $C_{G'}(d(u))=\{d(v)\mid u\to v\}$ for every $u\in V_G$ (which is a slight abuse of the terminologies ``represent'' and ``decoration'').

$\FAFA_2$ is obviously equivalent to the claim that every set can be injectively represented by an isomorphism-extensional APG, so we naturally conjecture that the following weakening of $\WFAFA_2$, namely $\WFAFA_2'$, is equivalent to $\WFAFA_2$: every set can be represented by an isomorphism-extensional APG.
It is trivial that $\WFAFA_2'$ implies $\AIE'$, so $\WFAFA_2'$ implies $\FT$.

A more graph-theoretic way to deal with $\WFAFA_2$ is to consider the following stronger assertion, which will be denoted by $\CWFAFA_2$ (``$\mathsf{C}$'' stands for ``combinatorial''): every extensional APG can be represented by an extensional and isomorphism-extensional APG. We also have $\CWFAFA_2'$ claiming that every extensional APG can be represented by an isomorphism-extensional APG. At least, we can prove their equivalence.

\begin{theorem}
    $\ZFF\vdash\CWFAFA_2\leftrightarrow\CWFAFA_2'$.
\end{theorem}
\begin{proof}
    We will construct an extensional and isomorphism-extensional APG $H$ to represent an arbitrary extensional APG $G$ as $\CWFAFA_2$ wishes.

    We first separate the well-founded part and the ill-founded part of $G$. Let $\WFP(G)=\{a\in G\mid G[a]$ be a well-founded APG$\}$, and $\IFP(G)=G\backslash \WFP(G)$. Abusing the notations, we will denote the subgraphs induced by $\WFP(G)$ and $\IFP(G)$ by the same symbols. Clearly, there is no edge from a vertex in $\WFP(G)$ to a vertex in $\IFP(G)$, and $C_{\IFP(G)}(v)\neq\emptyset$ for every vertex $v\in\IFP(G)$. 

    If $\IFP(G)=\emptyset$, $G$ is a well-founded extensional APG, so by Mostowski's collapsing lemma we have that $G$ is already an isomorphism-extensional APG. 
    Now suppose $\IFP(G)\neq\emptyset$, at least we have that $\WFP(G)$ is a well-founded extensional APG, so we just need to deal with $\IFP(G)$.
    
    $\IFP(G)$ is naturally an APG having $p_{G}$ as point, and $C_{\IFP(G)}(v)\neq\emptyset$ for every vertex $v\in\IFP(G)$.
    Let's consider the unfolding tree $T$ of $\IFP(G)$: a vertex in $T$ is a (finite) path starting from the point, and there is an edge from one vertex (path) to the other if the later path is obtained by adding one edge at the end of the former path. $T$ is naturally an APG having the 0-path $p_{G}$ as point, and $C_{T}(v)\neq\emptyset$ for every vertex $v\in T$. Moreover, $T$ is automatically extensional, and there is a natural decoration $\pi$ from $T$ to $\IFP(G)$: one that maps a path to the ending vertex of it. 

    Now we use $\CWFAFA_2'$ to get an isomorphism-extensional APG $K$ together with a decoration $\pi'$ from $K$ to $T$, and $K$ is automatically extensional as well. In fact, since there is no loop in $T$, there must be no loop in $K$ (followed from the definition of representation). Thus, isomorphism-extensionality of $K$ implies extensionality of $K$: if $C_K(a)=C_K(b)$, then a map which maps $a$ to $b$ and fixes the other vertices in $K[a]$ is an isomorphism from $K[a]$ to $K[b]$.

    Finally, we glue $K$ and $\WFP(G)$ together correctly to get an extensional and isomorphism-extensional $H$ representing $G$. In fact, we just take the (disjoint) union of the two sets of vertices, keep the edges from both graphs, and add $v\to w$ if $v\in K,w\in\WFP(G)$ and $\pi(\pi'(v))\to w$ in $G$. The verification is direct.
\end{proof}

We know that $\WFAFA_2$ is a weakening of $\FAFA_2$, but it is hard to prove $\CWFAFA_2$ from $\FAFA_2$. However, there is a powerful theorem that can help us with it, and we can then have that $\CWFAFA_2$ can be proved from $\AWF$ as well as a consequence.

\begin{theorem}
    $\GBF\vdash\GWFT\to\CWFAFA_2$.
\end{theorem}
\begin{proof}
    Given an extensional APG $G$, we may assume without loss of generality that $C_{G}(v)\neq\emptyset$ for every vertex $v\in G$ (followed from the proof of the previous theorem).

    By $\GWFT$, every set is assigned with a simple graph $\mathcal{G}_x=(V_x,E_x)$ that is unique up to isomorphism. Adding to $\mathcal{G}_x$ a new vertex $\alpha_x$ which is the least ordinal that doesn't belong to $V_x$ and a new edge $\{\alpha_x,a\}$ for every $a\in V_x$, we get $\mathcal{C}_x=(Q_x,R_x)$ which is a connected simple graph with at least two vertices that is unique up to isomorphism as well.

    Now we recursively construct an isomorphism-extensional APG $K$ and a decoration $\pi$ from $K$ to $G$. $K_0=\{(0,p_{G})\}$ having no edge and $(0,p_{G})$ as point, and $\pi((0,p_{G}))=p_{G}$. Then, $K_{n+1}=K_n\cup \{(1,x,c),(2,x,c,q),(3,x,c,d,r)\mid x\in K_n,c\in C_{G}(\pi(x)),q\in Q_x, d\in C_{G}(c), r\in R_x\}$ added with edges $x\to(1,x,c),x\to(2,x,c,q)$, $(2,x,c,q)\to(3,x,c,d,r)$ (if $q\in r$), and $\pi((1,x,c))=c,\pi((2,x,c,q))=c,\pi((3,x,c,d,r))=d$. 
    After $\omega$-steps, every vertex and edge in the graph $K=\bigcup_{n\in\omega}K_n$ is added at some unique step $n$, and we get the decoration $\pi$. By induction on $n$, every vertex in $K_n$ is accessible from $(0,p_{G})$, so $K$ is an APG. Every $x\in K$ has a child of the form $(1,x,c)$, and $(1,x,c)\neq(1,y,c')$ just follows from $x\neq y$, so $K$ is extensional.

    Now we prove that $K$ is isomorphism-extensional. For every $x\in K$, notice that if $u\neq v \in C_K(x)$ satisfy that $C_K(u)\cap C_K(v)\neq\emptyset$, then there exists some $c\in C_{G}(\pi(x))$ such that $u,v$ are both of the form $(2,x,c,q)$. This is because only these vertices will have common children of the form $(3,x,c,d,r)$ in our construction. Meanwhile, for every $(2,x,c,q)$, there exists some other $(2,x,c,q')$ which shares a common child with $(2,x,c,q)$ due to the connectedness of $\mathcal{C}_x$. So $U_{x}=\{(2,x,c,q)\mid c\in C_{G}(\pi(x)),q\in Q_x\}$ is definable from $x$, and then $R(U_{x})=\{N_{U_{x}}(w)\mid w\in K[x],|N_{U_{x}}(w)|=2\}$ (where $N_{U_{x}}(w)=\{u\in U_{x}\mid u\to w\}$) is definable from $x$. Notice that every connected component of $(U_{x},R(U_{x}))$ is of the form $(U_{x,c}=\{(2,x,c,q)\mid q\in Q_x\},R(U_{x,c}))$, and then is isomorphic to $(Q_x,R_x)$: the bijection is $(2,x,c,q)\mapsto q$, and that it is an isomorphism follows immediately from our construction. Thus, if $K[x]\cong K[y]$, then the isomorphism between them will induce an isomorphism between $\mathcal{C}_x$ and $\mathcal{C}_y$, so we have $x=y$.

\end{proof}

A similar argument shows that $\ZFF\vdash\DC\wedge\WFT\to\CWFAFA_2$, since in the proof above we only use $\GWFT$ once to get $\mathcal{G}_x$ and then $\mathcal{C}_x$ uniformly in the $\omega$-step construction. So $\WFT$ provides us with the existence of all such $\mathcal{G}_x$s, and $\DC$ provides us with the ability to do an $\omega$-step dependent construction.

Now, since $\FAFA_2$ gives rise to a definable class function for $\GWFT$ (just choose the canonical pictures), we have $\ZFF\vdash\FAFA_2\to\CWFAFA_2$. As a consequence, $\ZFF\vdash\AWF\to\CWFAFA_2$ as well: every extensional APG is isomorphic to an extensional APG in $\WF$, which will be represented by an extensional and isomorphism-extensional APG in $\WF$ because $\WF\vDash\CWFAFA_2$ (followed from $\WF\vDash\FAFA_2$). Then we know that $\CWFAFA_2$ can't prove $\FAFA_2$ or $\AWF$ because $\FAFA_2$ and $\AWF$ are independent.


\section{Open questions}

At the end of this paper, we leave some open questions:
\begin{itemize}
    \item We have proved several unprovability results in $\ZFF$. What will happen if we replace $\ZFF$ with $\ZFFC$?
    \item Can $\WFAFA_2$ prove $\CWFAFA_2$? Can $\WFAFA_2'$ prove $\WFAFA_2$ (we highly believe that the answer is yes)? And can $\FT$ prove $\WFAFA_2'$?
    \item Can we find another ordinary mathematical theorem that is equivalent with some set-theoretic principle below $\AC$ and $\AF$?
\end{itemize}

\printbibliography

\end{document}